\documentclass[a4paper,11pt]{amsart}

\usepackage{geometry}
\usepackage{amssymb,mathrsfs,amsmath,amsthm}
\usepackage{enumitem}
\usepackage{autobreak}
\allowdisplaybreaks
\usepackage{bm}

\usepackage{graphicx}
\usepackage{color}
\usepackage{booktabs}
\usepackage{hyperref}
\usepackage{cleveref}

\newtheorem{theorem}{Theorem}[section]
\newtheorem{lemma}{Lemma}[section]
\newtheorem{proposition}{Proposition}[section]

\newtheorem{remark}{Remark}[section]

\numberwithin{equation}{section}
\numberwithin{figure}{section}

\makeatletter
\@namedef{subjclassname@2020}{\textup{2020} Mathematics Subject Classification}
\makeatother

\begin{document}

\title{Overdetermined problem for optimal transportation}


\author{Qing Zhao}
\address{School of Mathematics and Shing-Tung Yau Center, Southeast University, Nanjing 211189, P. R. China}
\email{zhao-qing@seu.edu.cn}

\author{Feida Jiang$^*$}
\address{School of Mathematics and Shing-Tung Yau Center, Southeast University, Nanjing 211189, P. R. China; Shanghai Institute for Mathematics and Interdisciplinary Sciences, Shanghai 200433, P. R. China}
\email{jiangfeida@seu.edu.cn}


\date{\today}
\thanks{$^*$ Corresponding author: Feida Jiang.}



\begin{abstract}
In this paper, we establish symmetry results for solutions of overdetermined problems arising in optimal transportation. These problems involve the Monge-Amp\`{e}re equation with a Dirichlet boundary condition $u=0$ on $\partial \Omega$ and the natural boundary condition $Du(\Omega)=\Omega^{*}$. We show that symmetry holds when the target is the unit ball $B$. For more general target domains, including the cases $\Omega^*=\Omega$ and arbitrary $\Omega^{*}$, symmetry is retained under an additional volume constraint on $\Omega$ and a boundary condition on $|Du|$. Finally, we introduce a curvature-type
overdetermined problem for the Monge-Amp\`ere equation and obtain
ellipsoidal or spherical symmetry under an additional integral
normalization condition. Our proofs rely on distinct techniques in different contexts: optimal transport and convex analysis for some cases, integral identities and isoperimetric inequalities for others, and P-function methods for the remaining cases. As a byproduct, in the $\tau=2$ case, a new maximum principle for the $P$-function $\phi(x)=\sum_{k,l=1}^{n}{\frac{\partial{S_{\tau}(D^2{u})}}{\partial{u_{kl}}}u_{k}u_{l}}-2\binom{n-1}{\tau-1}\int_{0}^{u}{f^{\frac{\tau}{n}}(t)\,dt}$
is established for arbitrary positive and nondecreasing $f$, which is of independent interest.
\end{abstract}

\maketitle

\section{Introduction}\label{sec1}
In this paper, we consider the following overdetermined problem
\begin{equation}\label{1.4}
        \left\{
\begin{array}{ll}
\det{D^{2}u}=1&\text{in}~\Omega,\\
u=0&\text{on}~\partial\Omega,\\
Du(\Omega)=\Omega^{*},
\end{array}
\right.
\end{equation}
where the source domain $\Omega$ and the prescribed target domain $\Omega^{*}$ are bounded open domains in $\mathbb{R}^{n}$ with $C^{2}$ boundary, $Du$ and $D^2u$ denote the gradient vector and Hessian matrix of $u$, respectively. The motivation for studying this class of overdetermined problems originates from optimal transportation theory, and we refer to such problems as overdetermined problems for optimal transportation. In the following, we provide the necessary background on both optimal transportation and Serrin-type overdetermined problems.

\vspace{2mm}

\subsection{Optimal transportation} The optimal mass transportation problem, originally formulated by Monge \cite{bib34}, seeks a mapping between two mass distributions that minimizes the total cost. More precisely, let $\mu_0\in\mathcal{P}(\Omega)$ and $\mu_1\in\mathcal{P}(\Omega^{*})$ be two probability measures on domains $\Omega$ and $\Omega^{*}$, respectively, and let $c:\Omega\times{\Omega^{*}}\rightarrow[0,+\infty)$ be a cost function. The problem is to find a transport map $T:\Omega\rightarrow\Omega^{*}$ that minimizes
\begin{equation}
    \inf_{T}\Big\{\int_{\Omega}{c(x,T(x))}\,d{\mu_0}:T_{\sharp}\mu_0=\mu_1\Big\},
\end{equation}
where the measure denoted by $T_{\sharp}\mu_0=\mu_1$ is defined through $(T_{\sharp}\mu_0)(A)=\mu_0(T^{-1}(A))$ for every measurable set $A\subset{\Omega^{*}}$. The measure $T_{\sharp}\mu_0$ is called image measure or push-forward of $\mu_0$ through $T$.

The Monge problem is extremely challenging because its constraint is not closed under weak convergence. It remained unsolved until the 1940s, when Kantorovich \cite{bib21} introduced a suitable relaxation that made it possible to establish the existence and analyze its solutions. Kantorovich's formulation consists of minimizing the linear functional
\begin{equation}
    \inf_{\gamma}\Big\{\int_{\Omega\times{\Omega^{*}}}{c}\,d\gamma:\gamma\in\Pi(\mu_0,\mu_1)\Big\},
\end{equation}
where $\Pi(\mu_0,\mu_1)$ is  the nonempty convex set of all probability measures on $\Omega\times{\Omega^{*}}$ with marginals $\mu_0$ on $\Omega$ and $\mu_1$ on $\Omega^{*}$. More explicitly, $\gamma\in\Pi(\mu_0,\mu_1)$ if and only if $\gamma$ is a nonnegative measure satisfying
$$\gamma(A\times{\Omega^{*}})=\mu_0(A),~~\gamma(\Omega\times{B})=\mu_1(B),$$
for all measurable subsets $A\subset\Omega$ and $B\subset\Omega^{*}$.

The Kantorovich problem is a linear minimization problem with convex constraints. It is equivalent to maximizing the dual functional
$$\int_{\Omega}\phi(x)\,d\mu_{0}+\int_{\Omega^{*}}\psi(y)\,d{\mu_1}$$
 subject to the constraints $\phi(x)+\psi(y)\leq{c(x,y)}$.~Kantorovich's pioneering work \cite{bib21} established a deep connection between optimal transport, linear programming, and duality theory \cite{bib22}.~Using Kantorovich's duality theory, Brenier \cite{bib23} showed that when $c(x,y)=|x-y|^{2}$, the optimal transport map can be expressed as $T = D u$, where $u$  is a convex potential function. Assume further that $\mu_0$ and $\mu_1$ are absolutely continuous with respect to the Lebesgue measure, with respective densities $f_0$ and $f_1$. Let $\mu_1$ concentrated on $\Omega$, the interior of the set where $u$ is finite; then $T = D u$ pushes forward $\mu_0$ to $\mu_1$, and consequently $\mu_1$ is concentrated on $\Omega^{*}$. If $u$ is of class $C^2$ and $D u$ is injective on $\Omega$, then $u$ satisfies the Monge-Amp\`{e}re equation
$$\det D^2 u(x)=\frac{f_{0}(x)}{f_{1}(D u(x))} \quad {\rm in} \ \Omega,$$
and satisfies the natural boundary condition $$Du(\Omega)=\Omega^{*}.$$ 
For further details on optimal transportation, we refer the reader to \cite{bib20,bib33}. The Monge-Amp\`{e}re equation admits a rich family of invariants, we consider the simplest Monge-Amp\`{e}re equation $\det(D^{2}u)=1$, which gives rise to a natural boundary value problem
\begin{equation}\label{1.3}
    \left\{
\begin{array}{ll}
 \det{D^{2}u}=1&\text{in}~\Omega,\\
D u(\Omega)=\Omega^{*}.
\end{array}
\right.
\end{equation}
In two-dimensional case, Pogorelov \cite{bib24} proved the existence of Aleksandrov generalized solutions for the Monge-Amp\`{e}re equation on bounded convex domains; Delano\"{e} \cite{bib25} proved global regularity under uniform convexity and smooth boundary conditions, while Urbas \cite{bib26,bib27} gave a different proof and extended the results to more general equations. In higher dimensions, Brenier \cite{bib23} established the existence of weak solutions under very mild assumptions on $\Omega$ and $\Omega^*$; Caffarelli \cite{bib17} proved $C^{1,\alpha}$ global regularity when both domains are convex; Urbas \cite{bib28} proved $C^{2,\alpha}$ global regularity if both $\Omega$ and $\Omega^{*}$ are uniformly convex. Further relevant and recent studies are available in \cite{bibCLW,bibJT2014,bib31,bib32}.

In this paper, we consider the case that the mass on the boundary is restricted to transport along the normal direction. Under the additional restriction, we have to consider the overdetermined problem \eqref{1.4} instead of the second boundary value problem \eqref{1.3}.

\vspace{2mm}

\subsection{Serrin-type overdetermined problem}

Before stating our main result of the overdetermined problem \eqref{1.4}, we recall several historical results of Serrin-type overdetermined problem in this subsection. 

In a seminal paper in 1971, Serrin \cite{bib1} proved that the following overdetermined problem has a solution if and only if $\Omega$ is a ball:
$$
\left\{
\begin{array}{ll}
 \Delta{u}=1&\text{in}~\Omega,\\
u=0&\text{on}~\partial\Omega,\\
|D{u}|=c&\text{on}~\partial\Omega,
\end{array}
\right.
$$
where $c$ is a constant.
The proof relies on the moving plane method and the maximum principle. In a subsequent paper \cite{bib2}, Weinberger gave an alternative proof by applying the maximum principle to an auxiliary function ($P$-function in the sense of L.E. Payne \cite{bib30}) $\phi(x) = |D u|^2 - \frac{2}{n}u$. More precisely, he showed that $\phi(x)$ attains its maximum $c^2$ on $\partial\Omega$, and then, by means of a Poho\v{z}aev-type identity, concluded that $\phi(x)$ is constant throughout $\Omega$. This immediately implies that $u$ is radial and that $\Omega$ is a ball. Since these foundational contributions, numerous alternative proofs and generalizations have appeared, extending the results to overdetermined problems for both linear and nonlinear operators \cite{bib3,bib4,bib5,bib6,bib8}. In many of these works, the maximum principle applied to a suitable $P$-function plays an implicit but fundamental role. 

As the generalizations of Laplace equation, let us consider the following $k$-Hessian equation
\begin{equation}\label{e1}
    \left\{
\begin{array}{ll}
 S_{k}({D^{2}u})=f(u)&\text{in}~\Omega,\\
u=0&\text{on}~\partial\Omega,
\end{array}
\right.
\end{equation}
where $S_{k}$ is $k$-th elementary symmetric function, $f(u)$ is a positive smooth function. For $1\leq\tau\leq{k}$, Enache,~Marras and Porru \cite{bib9} introduces the following two classes of $P$-function
\begin{equation}\label{P-f}
\phi(x)=\sum_{k,l=1}^{n}{\frac{\partial{S_{\tau}(D^2{u})}}{\partial{u_{kl}}}u_{k}u_{l}}-2\binom{n-1}{\tau-1}\int_{0}^{u}{f^{\frac{\tau}{n}}(t)\,dt}
\end{equation}
and $$\psi(x)=\sum_{k,l=1}^{n}{\frac{\partial{S_{\tau}(D^2{u})}}{\partial{u_{kl}}}u_{k}u_{l}}-2\binom{n-1}{\tau-1}c_{0}^{\frac{\tau}{n}}u,$$
where $c_0=f(m_0)$,~$m_0=\min_{\Omega}{u}(x)$.
When $f=1$,~$\tau=2$ or $\tau=k=n$,~a maximum principle for $\phi(x)$ (or $\psi(x)$) was established by Chen,~Ma and Shi \cite{bib10}.~When $f=f(u)$,~$\tau=2$ or $\tau=k=n$,~a maximum principle for $\psi(x)$ was established by Feng and Shi \cite{bib11}.~When $f=f(u)$,~$\tau=k=n=2$,~a maximum principle for $\phi(x)$ has been established by Enache and Porru \cite{bib12}.~When $f=f(u)$,~$\tau=k=n$,~a maximum principle for $\phi(x)$ was established by Enache,~Marras and Porru \cite{bib9}.~When $f=f(u)$,~$\tau=1$,~Enache \cite{bib13} established the maximum principle for $\phi(x)$ and Enache, Marras and Porru \cite{bib9} established a similar maximum principle for the $\psi(x)$. To the best of our knowledge, a maximum principle for this particular
$P$-function $\phi$ has not been established for the case
$f=f(u)$ and $\tau=2$ under the present setting. We establish such a maximum principle in Lemma \ref{L1}.

The equation (\ref{e1}) becomes the standard Monge-Amp\`ere equation when $k=n$ and $f=1$. If the solution $u$ satisfies the boundary condition $H_{n-1}|Du|^{n+1}=c$ on $\partial\Omega$,~by applying the maximum principle to the $P$-function $\phi(x)$ ($\tau=n,f=1$) together with an entropy estimate from affine curvature flow, Brandolini et al. showed in \cite{bib14} that the overdetermined problem admits a solution if and only if $\Omega$ is an ellipsoid, where $H_{n-1}$ denotes the Gauss curvature of the boundary. However, many of the techniques developed do not rely on a direct application of the $P$-function.~If $u$ satisfies the boundary condition $H_{n-1}|Du|^{n-1}=1$ on $\partial\Omega$, via a straightforward application of Newton's inequality (Maclaurin inequalities), Enache and Porru \cite{bib13} showed that $\Omega$ is a ball. More generally,~for a fixed integer $1\leq{r}<n$,~if $u$ satisfies additional boundary condition $\frac{H_{n-1}}{H_{r-1}}|Du|^{r-1}=\binom{n-1}{r-1}^{-1}$ on $\partial\Omega$, Enache, Marras and Porru  \cite{bib9} proved that $\Omega$ is a ball by using the Newton's inequality, where $H_{r-1}$ denotes the $(r-1)$-th curvature of the boundary. Note that one can refer to Section \ref{Sec2.2} for the detailed definitions of the $k$-th curvature of the boundary. If the solution $u$ satisfies $|Du|=1$ on $\partial\Omega$, Brandolini et al. \cite{bib7} provided an alternative proof that avoids the $P$-function method, and  employ Newton's inequalities and the Poho\v{z}aev-type identity. They showed that a solution exists if and only if $\Omega$ is a ball. 

In this paper,we investigate symmetry for two distinct classes of
overdetermined problems, arising respectively from optimal transportation and nonstandard curvature boundary conditions. Our approach combines optimal transport, convex geometry,
quermassintegral inequalities, and the $P$-function method. First, for overdetermined problems associated with the transport conditions $Du(\Omega)=B$, $Du(\Omega)=\Omega$ and $Du(\Omega)=\Omega^{*}$, we
establish symmetry results for the Monge-Amp\`ere equation and their extensions to the $k$-Hessian equations. Second, for the nonstandard boundary condition $H_1|Du|^3=c$, we introduce an additional integral
normalization condition and obtain symmetry results for
$\det D^2u=1$. Finally, we extend the problem to
$\det D^2u=f(u)$ with $f$ positive and nondecreasing. The main new ingredient is a maximum principle for the associated $P$-function \eqref{P-f} in
the case $f=f(u)$ and $\tau=2$, which, together with the integral
normalization condition, yields the symmetry of both the domain and the solution.

\vspace{2mm}

\subsection{Main results}
We state our main theorems in this subsection. We consider the optimal transportation problem between the Lebesgue
measures on $\Omega$ and $\Omega^*$ with quadratic cost
$c(x,y)=|x-y|^2$. In this paper, we restrict our attention to the case
where the source and target densities are identically equal to $1$ on
their respective domains. By Brenier's
theorem, the optimal transport map is given by $T=Du$ for a convex potential $u$, and we assume that
$$
Du(\Omega)=\Omega^*,
\quad
u=0\quad\text{on }\partial\Omega.
$$
Accordingly, the Monge-Amp\`ere equation takes the form
$$
\det D^2u=1\quad\text{in }\Omega.
$$

The Dirichlet condition $u=0$ enforces the vanishing of the tangential component of $Du$, thereby restricting boundary transport to the normal direction. This strengthened boundary condition is standard in overdetermined problems for Monge-Amp\`{e}re equations, where it serves as a key tool for deriving symmetry results of the domain. Moreover, this normal constraint also has practical relevance in heat flow with absorbing boundaries \cite{bib41} and numerical discretizations of optimal transportation \cite{bib42}. 

We now state our main results, which concern the overdetermined problem for the Monge-Amp\`{e}re equation with Dirichlet and natural boundary conditions. We establish these results in three different situations regarding to the target domain $\Omega^*$, namely, when $\Omega^*$ is the unit ball $B$, when $\Omega^*$ coincides with $\Omega$, and when $\Omega^*$ is an arbitrary domain.

\begin{theorem}\label{T1.1}
Let $\Omega$ be a $C^2$ bounded convex domain and $u\in C^2(\overline{\Omega})$ a convex solution of
\begin{equation}
        \left\{
\begin{array}{ll}
\det{{D^{2}u}}=1&\text{in}~\Omega,\\
u=0&\text{on}~\partial{\Omega},\\
Du(\Omega)=B,
\end{array}
\right.
\end{equation}
where $B$ is the unit ball centered at the origin. Then, for some $x_0\in\mathbb{R}^n$, $u(x)=\dfrac{|x-x_0|^2-1}{2}$ and $\Omega$ is the unit ball centered at $x_0$.
\end{theorem}
\begin{theorem}\label{GTH1.2} Let $\Omega$ be a $C^2$ bounded convex domain, $|\Omega|\leq\omega_{n}$ and $u\in C^2(\overline{\Omega})$ be a convex solution of
\begin{equation}
        \left\{
\begin{array}{ll}
 \det{D^{2}u}=1&\text{in}~\Omega,\\
 u=0&\text{on}~\partial{\Omega},\\
|Du|\geq{1}&\text{on}~\partial{\Omega},\\
Du(\Omega)=\Omega,
\end{array}
\right.
\end{equation}
where $\omega_{n}$ is the volume of the $n$-dimensional unit ball. Then $u(x)=\dfrac{|x|^2-1}{2}$ and $\Omega$ is the unit ball centered at the origin. 
\end{theorem}
\begin{theorem}\label{GTH1.3} Let $\Omega$ be a $C^2$ bounded convex domain, $|\Omega|\leq\omega_{n}$ and $u\in C^2(\overline{\Omega})$ be a convex solution of
\begin{equation}
        \left\{
\begin{array}{ll}
 \det{D^{2}u}=1&\text{in}~\Omega,\\
 u=0&\text{on}~\partial{\Omega},\\
|Du|\geq{1}&\text{on}~\partial{\Omega},\\
Du(\Omega)=\Omega^{*},
\end{array}
\right.
\end{equation}
where $\omega_{n}$ is the volume of the $n$-dimensional unit ball. Then, for some $x_0\in\mathbb{R}^n$, $u(x)=\dfrac{|x-x_0|^2-1}{2}$ and $\Omega$ is the unit ball centered at $x_0$, $\Omega^*$ is the unit ball centered at the origin. 
\end{theorem}


Inspired by \cite{bib14}, we first consider the Monge-Amp\`{e}re equation
with the nonstandard boundary condition $H_1|Du|^3=c$. By imposing an additional integral normalization
condition, we obtain a symmetry result for the corresponding
overdetermined problem.

\begin{theorem}\label{TH2}
    Let $\Omega\subset\mathbb{R}^{n}$ be a $C^{2}$ bounded convex domain.~Assume that $u\in{C^{2}(\overline{\Omega})}$ is a convex solution to the following equation:
    \begin{equation}
        \left\{
\begin{array}{ll}
 \det{D^{2}u}=1&\text{in}~\Omega,\\
u=0&\text{on}~\partial\Omega,\\
H_{1}|Du|^{3}=c&\text{on}~\partial\Omega,
\end{array}
\right.
    \end{equation}
    where $c>0$ is a constant and $H_{1}$ denotes $(n-1)$ times the mean curvature of $\partial\Omega$. Assume, in addition, that
\begin{equation}\label{1.10}
\frac1{|\Omega|}
\int_\Omega(-u)\,dx
\geq
\frac{c}{(n-1)(n+2)}.
\end{equation}
Then domain $\Omega$ is an ellipse for $n=2$ or a ball for $n\geq{3}$ and $u$ is elliptically symmetric for $n=2$ or spherically symmetric for $n\ge 3$.
\end{theorem}

We now introduce a more general
overdetermined problem for the equation $\det D^2u=f(u)$, where $f$ is a positive nondecreasing function.
\begin{theorem}\label{TH2-new}
    Let $\Omega\subset\mathbb{R}^n$ be a bounded convex domain with
$C^2$ boundary. Assume that $f\in C^2(\mathbb{R})$ is positive and
nondecreasing. Let
$u\in C^4(\Omega)\cap C^2(\overline{\Omega})$ be a convex solution of
\begin{equation}
\left\{
\begin{array}{ll}
\det D^2u=f(u) & \text{in }\Omega,\\
u=0 & \text{on }\partial\Omega,\\
H_1|Du|^3=c & \text{on }\partial\Omega,
\end{array}
\right.
\end{equation}
and assume that $\frac{\lambda_{\max}(x)}{\lambda_{\min}(x)}
\leq
\left(\frac{3n}{4}\right)^{\frac{n}{2(n-1)}}$ for all $x\in\Omega$, where $\lambda_{\min}(x)$ and $\lambda_{\max}(x)$ denote the smallest
and largest eigenvalues of $D^2u(x)$, respectively, $c>0$ is a
constant, and $H_1$ denotes the first curvature of $\partial\Omega$,
that is, $(n-1)$ times its mean curvature. Assume, in addition, that
\begin{equation}\label{1.10-new}
\frac{n-1}{|\Omega|}
\int_\Omega
\left[
-nu f(u)^{\frac2n}
-2\int_0^u f(t)^{\frac2n}\,dt
\right]dx
\geq c.
\end{equation}
Then $\Omega$ is an ellipse when $n=2$ and a ball when $n\geq3$.
Moreover, $u$ is elliptically symmetric when $n=2$ and spherically
symmetric when $n\geq3$.
\end{theorem}
\begin{remark}
Conditions \eqref{1.10} and \eqref{1.10-new} hold with equality for the elliptically symmetric solution when $n=2$, and for the spherically symmetric solution when $n\ge3$. It is precisely the integral condition that allows us to exclude the possibility that \(\varphi<c\) somewhere in \(\Omega\), and hence force \(\varphi\equiv c\), which is the key step in the proof of Theorem \ref{TH2} and Theorem \ref{TH2-new}.
\end{remark}


\vspace{2mm}

\subsection{Organization of the paper}

The structure of the paper is as follows.

In Section \ref{sec2}, we introduce some basic notation of $k$-th elementary symmetric functions and $k$-Hessian operators, and recall a few facts concerning curvatures and quermassintegrals of smooth domain. The isoperimetric inequalities for quermassintegrals and Legendre-Fenchel transforms of convex functions are also presented. 

In Section \ref{sec3}, we establish the overdetermined problem in
Theorem \ref{TH3.1} for the $k$-Hessian equations by an integral
method, treating the cases $k=1$, $2\leq k\leq n-1$, and $k=n$
separately. Theorem \ref{T1.1} is then obtained as the special case
$k=n$ of Theorem \ref{TH3.1}. Moreover, we provide an alternative proof for Theorem \ref{T1.1}, which uses some simple tools of convex analysis. 

In Section \ref{sec4}, we first establish symmetry results for the
Monge-Amp\`ere equation in Theorems \ref{GTH1.2} and
\ref{GTH1.3} by an optimal transport and convex-geometric argument.
We then extend these results to the more general $k$-Hessian
equations for $1\leq k\leq n-1$, as stated in Theorems
\ref{T5.1} and \ref{T5.2}. The latter results are proved by means
of the isoperimetric inequalities for quermassintegrals.

In Section \ref{sec5}, we further prove Theorem \ref{TH2} and \ref{TH2-new} via the $P$-function method. A new version of maximum principle for $P$-function \eqref{P-f} when $f=f(u)$ and $\tau=2$ is given by Lemma \ref{L1}. Finally, in Section \ref{sec6}, we discuss possible directions for future research.

\vspace{3mm}

\section{Notation and Preliminaries}\label{sec2}
In this section, we introduce the notation and collect the preliminary propositions that will be employed in the subsequent proofs. The presentation is organized into four subsections, which cover, respectively, symmetric functions and Hessian operators, mean curvatures and quermassintegrals, isoperimetric inequalities for quermassintegrals, and the conjugate of a convex function.

\subsection{Symmetric functions and Hessian operators}
We denote by $A=(a_{ij})$ a matrix in the space $\mathcal{S}_{n}$ of the real symmetric $n\times{n}$ matrices, and by $(\lambda_{1},\cdots,\lambda_{n})$ its eigenvalues. For $k\in\{1,\cdots,n\}$, the $k$-th elementary symmetric functions of $A$ is 
$$S_{k}(A)=S_{k}(\lambda_{1},\cdots,\lambda_{n})=\sum_{1\leq{i_{1}}<\cdots<i_{k}\leq{n}}{\lambda_{i_{1}}\cdots\lambda_{i_{k}}}.$$

The operator $S^{\frac{1}{k}}_{k}$, for $k=1,\cdots,n$, is homogeneous of degree $1$ and concave, if it is restricted to the cone
$$\Gamma_{k}=\left\{A\in\mathcal{S}_{n}:S_{i}(A)\geq{0}~\text{for}~ i=1,\cdots,k\right\}.$$
Denoting by 
$$S_{k}^{ij}(A)=\frac{\partial}{\partial{a_{ij}}}S_{k}(A),$$
the Euler identity for homogeneous functions gives
$$S_{k}(A)=\frac{1}{k}S^{ij}_{k}(A)a_{ij}.$$
Here and throughout this paper, the Einstein summation convention over repeated indices is employed.

Assume that $A\in\Gamma_{n}$ (that is all the eigenvalues are nonnegative); then the following inequalities, known as Maclaurin inequalities, hold:
\begin{equation}\label{e7}
    \frac{S_{1}(A)}{\binom{n}{1}}\geq\cdots\geq\left(\frac{S_{k}(A)}{\binom{n}{k}}\right)^{\frac{1}{k}}\geq\cdots\geq\left(\frac{S_{n}(A)}{\binom{n}{n}}\right)^{\frac{1}{n}},
\end{equation}
and equality at any stage in (\ref{e7}) implies $\lambda_{1}=\lambda_{2}=\cdots=\lambda_{n}$, see \cite{bib29}.

Let $\Omega$ be an open subset of $\mathbb{R}^{n}$ and let $u\in{C^{2}(\Omega)}$. The $k$-Hessian operator $S_{k}(D^{2}u)$ is defined as the $k$-th elementary symmetric function of $D^{2}u$. Notice that
$$S_{1}(D^2{u})=\Delta{u}\quad\text{and}\quad {S_{n}(D^{2}u)=\det{D^{2}u}}.$$
For $ k > 1 $, the $k$-Hessian operators are fully nonlinear and, in general, not elliptic, unless restricted to the class of $k$-convex functions
$$
\Phi_k^2(\Omega) = \left\{ u \in C^2(\Omega) : S_i(D^2u) \geq 0 \text{ in } \Omega, i = 1, 2, \cdots, k \right\}.
$$
Notice that $\Phi_n^2(\Omega)$ coincides with the class of $C^2(\Omega)$ convex functions. A direct computation yields that $(S_k^{ij}(D^2u), \cdots, S_k^{nj}(D^2u))$ is divergence free, that is
 \begin{equation}\label{F1}
\frac{\partial}{\partial x_i} S_k^{ij}(D^2u) = 0, \quad \forall j=1,\cdots,n.
\end{equation}
Hence, $S_k(D^2u)$ can be written in the following divergence form:
\begin{equation}\label{F2}
S_k(D^2u) = \frac{1}{k} S_k^{ij}(D^2u) u_{ij} = \frac{1}{k} \left( S_k^{ij}(D^2u) u_{j} \right)_i,
\end{equation}
where subscripts $i$ and $j$ denote partial differentiation with respect to the corresponding variables $x_i$ and $x_j$, respectively.

\vspace{2mm}

\subsection{Mean curvatures and quermassintegrals}\label{Sec2.2}
Let $ D $ be a bounded connected domain of $ \mathbb{R}^n$ of class $C^2 $ having principal curvatures $\kappa = (\kappa_1, \cdots, \kappa_{n-1}) $ and outer unit normal $ \nu_x $. For $ k = 1, \cdots, n-1 $, we define the $k$-th curvature of $ \partial D $ by
$$
H_k(\partial D) = S_k(\kappa_1, \cdots, \kappa_{n-1}).
$$
Moreover, we set
$$
H_0 = S_0 \equiv 1 \quad {\rm and} \quad\ H_n \equiv 0.
$$
For example, $H_1 $ is equal to $n-1$ times the mean curvature of $ \partial D$, while $H_{n-1}$ is the Gauss curvature of $\partial D$. A domain $D$ is said to be $k$-convex, with $k\in\{1,\cdots,n-1\}$, if $H_{j}\geq{0}$ for $j=1,\cdots,k$ at every point $x\in\partial{D}$.

For $k = 1, \cdots, n$, the quermassintegral $ W_k(D)$ is defined by
\begin{equation}\label{F4}
W_k(D) = \frac{1}{n \binom{n-1}{k-1}} \int_{\partial D} H_{k-1}.
\end{equation}
For $k=0$, we take $W_{0}(D)=|D|$. When $k=1$, we get $W_{1}(D)=\frac{|\partial{D}|}{n}$, see \cite{bib35,bib36}. 

Next, we recall the following identities, known in differential geometry and  
in the theory of convex bodies as Minkowskian integral formulae (see \cite{bib36,bib37}, for  
instance): 
\begin{equation}\label{F5}
\int_{\partial D} \langle x, \nu_x \rangle \, H_{k-1} = (n-k+1) \binom{n}{k-1} W_{k-1}(D), 
\end{equation}
where $\langle \cdot, \cdot \rangle$ is the Euclidean inner product.

Suppose $u \in C^2 $ and let $ t $ be a regular value of $ u $. Define the sublevel set $ L = \{ x : u(x) \le t \}$. Denote by $H_k $ the $ k $-th curvature of the boundary $\partial L $ at a point $ x $. A classical result states that
$$
S_1(D^2 u) = H_1 |Du| + \frac{u_{ij} u_{i}u_{j}}{|Du|^2},
$$
which implies that the Laplacian $ \Delta u $ at a point depends only on the derivatives of $ u $ along the direction of steepest descent through that point and on the mean curvature $ H_1/(n-1) $ of the level surface passing through the point. For a general index $1 \le k \le n $, a direct computation gives
\begin{equation}\label{F6}
S_k(D^2 u) = H_k |Du|^k + \frac{S_k^{ij} u_{i} u_{l} u_{lj}}{|Du|^2}.
\end{equation}
Moreover, the following pointwise identity holds for every $ 1 \le k \le n $ (see \cite{bib38,bib15}):
\begin{equation}\label{F7}
H_{k-1} = \frac{S_k^{ij}(D^2 u) u_{i} u_{j}}{|Du|^{k+1}}. 
\end{equation}

In order to prove the Theorem \ref{T1.1}, we need the following Poho\v{z}aev-type identity for the Hessian operators (see \cite{bib39}):

\begin{proposition}\label{P2.1}
Let $f \in C^1(\mathbb{R})$ be a nonnegative function and let $F(u) = \int_u^0 f(s) \, ds$.  
If $u \in C^2(\Omega) \cap C^1(\overline{\Omega})$ is a solution to the problem  

$$
\begin{cases} 
S_k(D^2u) = f(u) & \text{in } \Omega, \\ 
u = 0 & \text{on } \partial\Omega, 
\end{cases}
$$
in a $C^2$ domain $\Omega \subset \mathbb{R}^n$, then  

\begin{equation}\label{F8}
\frac{n - 2k}{k(k + 1)} \int_\Omega S_k^{ij} u_i u_j + \frac{1}{k + 1} \int_{\partial\Omega} \langle x, \nu_x \rangle |Du|^{k+1} H_{k-1} = n \int_\Omega F(u). 
\end{equation}
\end{proposition}

\vspace{2mm}

\subsection{Isoperimetric inequalities for quermassintegrals}With $\Omega\subset\mathbb{R}^{n}$ being bounded convex domain with $C^{2}$ boundary, we can define the $m$-mean radius of $\Omega$, $\zeta_m(\Omega)$, by
$$\quad \zeta_m(\Omega) = \left( \frac{W_{n-m}(\Omega)}{\omega_n} \right)^{\frac{1}{m}}, \quad m = 1, \cdots,n,
$$
where $\omega_{n}$ is the volume of the $n$-dimensional unit ball. In particular we have
$$
\zeta_n(\Omega) = \left( \frac{|\Omega|}{\omega_n} \right)^{\frac{1}{n}},
$$
and
$$
\zeta_{n-1}(\Omega) = \left( \frac{|\partial\Omega|}{n\omega_n} \right)^{\frac{1}{n-1}},
$$
for arbitrary $\Omega$. The isoperimetric inequalities in \cite{bib36,bib40} assert that
\begin{equation}\label{F2.9}
\zeta_l(\Omega) \leq \zeta_m(\Omega),
\end{equation}
provided $1\leq{m}\leq{l}\leq{n}$, and include the classical isoperimetric inequality when
$l=n$ and $m=n-1$. Equality holds if and only if $\Omega$ is a ball.

\vspace{2mm}

\subsection{The conjugate of a convex function}
Let $\Omega \subset \mathbb{R}^n$ be a bounded convex open set and suppose $u \in C(\overline{\Omega})$ is convex. The conjugate (or Legendre-Fenchel transform) of $u$ is given by
\begin{equation}\label{F9}
u^*(\xi) = \max \left\{ \langle \xi, x \rangle - u(x) : x \in \overline{\Omega} \right\} \qquad \text{for } \xi \in \mathbb{R}^n. 
\end{equation}
Since $u^*$ is defined as the supremum of a family of affine functions, it is automatically convex. If $u$ is strictly convex and of class $C^1$ on $\Omega$, then $u^*$ is $C^1$ on the image set
$$
\Omega^* = Du(\Omega) = \{ Du(x) : x \in \Omega \}.
$$
Moreover, for any $\xi \in \Omega^*$, the gradient $Du^*(\xi)$ coincides with the unique point $x \in \overline{\Omega}$ such that $\xi = Du(x)$; in other words,
\begin{equation}\label{F10}
Du^* = (Du)^{-1} \qquad \text{in}~\Omega^*.
\end{equation}
This implies that $u \in C^2_+(\Omega)$ (i.e., $u$ is of class $C^2$ with $D^2u$ positive definite on $\Omega$) if and only if $u^* \in C^2_+(\Omega^*)$; moreover, when this holds, the following relations are satisfied:

\begin{equation}\label{F11}
D^2u^*(\xi) = (D^2u(x))^{-1} \quad \text{and} \quad u^*(\xi) + u(x) = \langle \xi, x \rangle,
\end{equation}
where $x = Du^*(\xi)$ and $\xi = Du(x)$.
For further properties of convex conjugates, see \cite{bib19} for instance.

\vspace{3mm}

\section{The case when $D u(\Omega)=B$}\label{sec3}
In this section, we consider the special case where the target domain
is $\Omega^*=B$, with $B$ denoting the unit ball centered at the origin. We first establish a general symmetry result for the
$k$-Hessian equation by an integral method, treating the cases
$k=1$, $2\leq k\leq n-1$, and $k=n$ separately. We then give an
alternative proof of Theorem \ref{T1.1} based on convex analysis.
\subsection{Integral method}
The following theorem gives the $k$-Hessian generalization, whose
endpoint case $k=n$ yields Theorem \ref{T1.1}.
\begin{theorem}\label{TH3.1}
Let $k\in\{1,\cdots,n\}$, $\Omega$ be a $C^2$ bounded convex domain and $u\in C^2(\overline{\Omega})$ be a strictly convex solution of
\begin{equation}\label{EF3.7}
        \left\{
\begin{array}{ll}
 S_{k}({D^{2}u})=\binom{n}{k}&\text{in}~\Omega,\\
u=0&\text{on}~\partial{\Omega},\\
Du(\Omega)=B,
\end{array}
\right.
\end{equation}
where $B$ is the unit ball centered at the origin. Then, for some $x_0\in\mathbb{R}^n$, $u(x)=\dfrac{|x-x_0|^2-1}{2}$ and $\Omega$ is the unit ball centered at $x_0$.
\end{theorem}

Since $u$ is strictly convex, the map
$Du:\overline{\Omega}\to\overline{B}$ is a homeomorphism mapping
$\partial\Omega$ onto $\partial B$. Hence, $|Du|=1$ on $\partial\Omega$. For $2\leq k\leq n-1$, we follow the integral method in
\cite{bib1} and include the details here for completeness. For the case 
$k=n$, we adopt a different approach from that in \cite{bib7}, which is made possible by the special transport condition $Du(\Omega)=B$.
\begin{proof}[Proof of Theorem \ref{TH3.1}]
\textbf{(1) Case $k=1$.} 
Since $S_1(D^2u)=\Delta u$, equation \eqref{EF3.7} becomes
$$
 \left\{
\begin{array}{ll}
 \Delta u=n&\text{in}~\Omega,\\
u=0&\text{on}~\partial{\Omega},\\
|Du|=1&\text{on}~\partial{\Omega}.
\end{array}
\right.
$$
This is the classical Serrin overdetermined problem, and the conclusion follows directly from Serrin's theorem \cite{bib1}.

\noindent
\textbf{(2) Case $2\le k\le n-1$.} By (\ref{F2}), (\ref{F7}) and using the fact that $|Du|=1$ on $\partial\Omega$, we have
\begin{equation}\label{F12}
|\Omega|=\int_{\Omega}\frac{S_{k}(D^{2}u)}{\binom{n}{k}}=\int_{\Omega}\frac{(S^{ij}_{k}(D^{2}u)u_{j})_{i}}{k\binom{n}{k}}=\int_{\partial\Omega}\frac{S^{ij}_{k}(D^{2}u)u_{j}u_{i}}{k\binom{n}{k}}=\int_{\partial\Omega}\frac{H_{k-1}}{k\binom{n}{k}}.
\end{equation}
By (\ref{e7}), (\ref{F2}), (\ref{F7}) and $|Du|=1$ on $\partial\Omega$, we have
\begin{equation}\label{F13}
\begin{aligned}
|\Omega| \leq \int_{\Omega} \frac{S_{k-1}(D^2 u)}{(k-1) \binom{n}{k-1}} = & \frac{1}{(k-1) \binom{n}{k-1}} \int_{\partial \Omega} S_{k-1}^{ij}(D^{2}u) u_i u_j \\
= & \frac{1}{(k-1) \binom{n}{k-1}} \int_{\partial \Omega} H_{k-2}
= W_{k-1}(\Omega).
\end{aligned}
\end{equation}
Notice that (\ref{F1}), (\ref{F2}) and $u=0$ on $\partial\Omega$, we get
\begin{equation}\label{F14}
\begin{aligned}
    \frac{n-2k}{k(k+1)} \int_{\Omega} S_k^{ij}(D^{2}u) u_{i}u_{j}
&= \frac{n-2k}{k(k+1)} \int_{\partial\Omega} S_k^{ij}(D^{2}u) u_i{u} - \frac{n-2k}{k(k+1)} \int_{\Omega} ( S_k^{ij}(D^{2}u) u_{i})_{j} u\\
&= -\frac{n-2k}{k(k+1)} \int_{\Omega} S_k^{ij}(D^{2}u) u_{ij} u\\
&= -\frac{n-2k}{k+1} \int_{\Omega} S_k(D^{2}u) u\\
&= -\frac{n-2k}{(k+1)} \binom{n}{k} \int_{\Omega} u.
\end{aligned}
\end{equation}
Now, by (\ref{F8}), (\ref{F5}), (\ref{F14}) and $|Du|=1$ on $\partial\Omega$, we obtain
\begin{equation}\label{F15}
    \int_{\Omega}{-u}=\frac{1}{k(n+2)\binom{n}{k}}\int_{\partial\Omega}{\langle{x,\nu_{x}}\rangle} H_{k-1}=\frac{{(n-k+1)}\binom{n}{k-1}}{k(n+2)\binom{n}{k}}W_{k-1}(\Omega)=\frac{1}{n+2}W_{k-1}(\Omega).
\end{equation}
Using (\ref{EF3.7}), (\ref{F2}), (\ref{F6}), (\ref{F7}) and again the fact that $|Du| = 1$ on $\partial \Omega$, we have
\begin{equation}\label{F16}
\begin{aligned}
\int_{\Omega} |Du|^2 = \int_{\Omega} |Du|^2 \frac{S_k(D^2u)}{\binom{n}{k}}
&= \frac{1}{k\binom{n}{k}} \int_{\Omega} |Du|^2 \left( S_k^{ij}(D^2u) u_j \right)_i\\
&= \frac{1}{k\binom{n}{k}} \left[ -2 \int_{\Omega} S_k^{ij}(D^2{u}) u_i u_l u_{lj} + \int_{\partial \Omega} |Du|^2 S_k^{ij}(D^{2}u) u_i u_j \right]\\
&= -\frac{2}{k\binom{n}{k}} \int_{\Omega} \left[ S_k(D^2u) |Du|^2 - H_k |Du|^{k+2} \right] + \frac{1}{k\binom{n}{k}} \int_{\partial \Omega} H_{k-1}\\
&=-\frac{2}{k} \int_{\Omega}|Du|^2+\frac{2}{k\binom{n}{k}} \int_{\Omega} H_k |Du|^{k+2} + \frac{1}{k\binom{n}{k}} \int_{\partial \Omega} H_{k-1}.
\end{aligned}
\end{equation}
The above equality gives
\begin{equation}\label{F17}
    \int_{\Omega} H_k |Du|^{k+2} = \binom{n}{k} \left( \frac{k+2}{2} \right) \int_{\Omega} |Du|^2 - \frac{1}{2} \int_{\partial \Omega} H_{k-1}.
\end{equation}
The equality (\ref{F17}), in conjunction with (\ref{F12}), (\ref{F13}) and the following inequality
$$
\int_{\Omega} |Du|^2 = \int_{\Omega} (-u) \Delta u \geq n \int_{\Omega} (-u) = \frac{n}{n+2} W_{k-1}(\Omega),
$$
yields
$$
\int_{\Omega} H_k |Du|^{k+2} \geq \frac{k+1}{n+2} \binom{n}{k+1} W_{k-1}(\Omega).
$$
Combining the last inequality with (\ref{F7}), (\ref{e7}) and (\ref{F15}), we get
$$\begin{aligned}
W_{k-1}(\Omega) &\leq \frac{n+2}{k+1} \binom{n}{k+1}^{-1} \int_{\Omega} H_k |Du|^{k+2}\\
&= (n+2) \binom{n}{k+1}^{-1} \int_{\Omega} (-u) S_{k+1}(D^2 u)\\
&\leq (n+2) \int_{\Omega} (-u)= W_{k-1}(\Omega).
\end{aligned}
$$
This implies that $S_{k+1}(D^2{u}) =\binom{n}{k+1}$. Hence, the Hessian matrix $D^2{u}$ has all equal
eigenvalues at every point of $\Omega$. This fact, together with the equation (\ref{EF3.7}), implies
that $D^2{u}$ is the identity matrix. Thus there exists $x_0 \in \mathbb{R}^n$ such that
$$
u(x) = \frac{1}{2} |x - x_0|^2 + C.
$$
From $u = 0$ on $\partial\Omega$ and $|Du|=1$ on $\partial\Omega$, we get
$$
 |x-x_0|=1\quad\text{on}~\partial\Omega\quad\text{and}\quad C=-\frac{1}{2},
$$
and therefore
$$
u(x) = \frac{|x - x_0|^2 - 1}{2}
$$
and $\Omega$ is the unit ball centered at $x_0$.

\noindent
\textbf{(3) Case $k=n$.} In this case, $
\det D^2u=1$ in $\Omega$. Since $u$ is strictly convex and $Du(\Omega)=B$, we have
$$
Du(\partial\Omega)=\partial B\quad\text{and}\quad
|Du|=1\quad\text{on }\partial\Omega.
$$
First, by the change of variables formula,
\begin{equation}\label{3.8}
|\Omega|=\int_\Omega \det D^2u=|B|=\omega_n.
\end{equation}
The identity \eqref{F15} remains valid for $k=n$, since its derivation does not use the restriction $k\leq n-1$. Thus,
\begin{equation}\label{3.9}
\int_\Omega(-u)=
\frac{1}{n+2}W_{n-1}(\Omega).
\end{equation}
Since $u$ is strictly convex and $\det D^2u=1$, the
 Maclaurin inequalities (\ref{e7}) gives
$$\frac{\Delta u}{n} \geq (\det D^2u)^{1/n} = 1,$$
that is,
$$
\Delta u \geq n.
$$
On the other hand, since $u = 0$ on $\partial\Omega$, integration by parts gives
$$
\int_\Omega |Du|^2 = -\int_\Omega u \Delta u.
$$
Combining $\Delta u \geq n$ and $u < 0$ in $\Omega$, we obtain
$$
\int_\Omega |Du|^2 \geq n \int_\Omega (-u).
$$
Using (\ref{3.9}), we obtain
\begin{equation}\label{3.10}
\int_\Omega |Du|^2 \geq \frac{n}{n+2} W_{n-1}(\Omega).
\end{equation}
Let $\xi\in B$, since $Du(\Omega)= B$ and $\det D^2 u = 1$, the change of variables gives
\begin{equation}\label{3.11}
\int_{\Omega} |Du|^2 = \int_B |\xi|^2
= \frac{n}{n+2} \omega_n.
\end{equation}
Therefore, by (\ref{3.10}), we get
$$
\frac{n}{n+2} \omega_n \geq \frac{n}{n+2} W_{n-1}(\Omega),
$$
hence
\begin{equation}\label{3.12}
W_{n-1}(\Omega) \leq \omega_n.
\end{equation}
From (\ref{3.8}), we have
$$
\zeta_n(\Omega) = \left( \frac{|\Omega|}{\omega_n} \right)^{1/n} = 1.
$$
Using the quermassintegral isoperimetric inequality (\ref{F2.9}), taking $
l = n$, $m = 1$, we obtain
$$
\zeta_n(\Omega) \leq \zeta_1(\Omega).
$$
Furthermore,
$$
\zeta_1(\Omega) = \frac{W_{n-1}(\Omega)}{\omega_n},
$$
hence
$$
1 \leq \frac{W_{n-1}(\Omega)}{\omega_n}.
$$
That is,
\begin{equation}\label{3.13}
W_{n-1}(\Omega) \geq \omega_n.
\end{equation}
Combining (\ref{3.12}) and (\ref{3.13}),
$$
W_{n-1}(\Omega) = \omega_n.
$$
Thus equality holds in the isoperimetric inequality, and therefore $\Omega$ is a ball.
From (\ref{3.9}),
$$
\int_{\Omega} (-u) = \frac{\omega_n}{n+2}.
$$
From (\ref{3.11}),
$$
\int_{\Omega} |Du|^2 = \frac{n}{n+2} \omega_n = n \int_{\Omega} (-u).
$$
But we already have
$$
\int_{\Omega} |Du|^2 = \int_{\Omega} (-u) \Delta u.
$$
Therefore,
$$
\int_{\Omega} (-u)(\Delta u - n) = 0.
$$
Since $-u > 0 $ in $\Omega$ and $\Delta u - n \geq 0$, we
conclude that $\Delta u = n$ in $\Omega$. Thus
$$
\frac{\Delta u}{n} = 1 = (\det D^2 u)^{1/n}.
$$
Equality holds in the above inequality, hence all eigenvalues of $D^2 u$ are equal. Since $
\det D^2 u = 1$, each eigenvalue equals 1, that is, $D^2 u = I$. Thus there exists $x_0 \in \mathbb{R}^n$ such that
$$
u(x) = \frac{1}{2} |x - x_0|^2 + C.
$$
From $u=0$ on $\partial\Omega$, we obtain that $\Omega$ is a ball centered at $x_0$. From \eqref{3.8}, we obtain that $\Omega$ is the unit ball centered at $x_0$.
\end{proof}

The proof of Theorem \ref{TH3.1} for $k=n$ also yields a proof of Theorem \ref{T1.1}. Indeed, since $u$ is a convex $C^{2}$ solution of 
$$
\det D^{2}u=1\quad\text{in }\Omega,$$
the Hessian $D^{2}u$ is positive semidefinite, and its determinant is strictly positive. Hence $D^{2}u$ is positive definite in $\Omega$, which implies that $u$ is strictly convex. Therefore, the strict convexity assumption required in Theorem \ref{TH3.1} is automatically satisfied. 

\begin{remark}
In this paper, we consider $C^2$ convex solutions of Monge-Amp\`ere equation with smooth positive right-hand side. Due to Remark 3.8 and Corollary 4.11 in \cite{bibFig}, such convex solutions with $C^2$ regularity are strictly convex.
\end{remark}

\vspace{2mm}

\subsection{Convex analysis method}
We give an alternative proof of Theorem \ref{T1.1} based on the  Legendre-Fenchel transform.

\begin{proof}[Proof of Theorem \ref{T1.1}] First,~we observe that
$$|\Omega|=\int_{\Omega}{\det{D^{2}u}}=|B|.$$
Let $v$ be the conjugate function of $u$. Since $u$ is of class $C^{2}$ and $D^{2}u$ is positive definite in $\Omega$,~it follows that
$$D^{2}v(\xi)=(D^{2}u(x))^{-1}\quad \text{and}\quad\det{D^{2}v}=\frac{1}{\det{D^{2}u}}=1,$$
and $v$ solves the following problem
\begin{equation}\label{e20}
        \left\{
\begin{array}{ll}
 \det{D^{2}v}=1&\text{in}~B,\\
v(\xi)=\frac{\partial{v}}{\partial\nu_{\xi}}&\text{on}~\partial{B},\\
Dv(B)=\Omega,
\end{array}
\right.
    \end{equation}
where the second equality of \eqref{F11} and $u=0$ on $\partial\Omega$ are used, and $\nu_{\xi}$ denotes the outer unit normal to $\partial{B}$. Note that we have used the fact that $Du(\partial\Omega)=\partial{B}$. By the  Maclaurin inequalities (\ref{e7}), we obtain
\begin{equation}\label{e21}
  n|B|=n\int_{B}{(\det{D^{2}v})^{\frac{1}{n}}}\leq\int_{B}{\Delta{v}}=\int_{\partial{B}}{\frac{\partial{v}}{\partial\nu_{\xi}}}=\int_{\partial{B}}{v}.
\end{equation}
Moreover, by the second equality of \eqref{F11}, the following holds
$$u(Dv(\xi))=\langle{Dv(\xi),~\xi}\rangle-v(\xi).$$
Therefore, we have
\begin{equation}\label{e22}
    \begin{aligned}
        \int_{\Omega}{u}=\int_{B}{u(Dv(\xi))}&=\int_{B}{\langle{Dv(\xi),~\xi}\rangle-v(\xi)}\\
        &=\int_{B}{[\text{div}(v(\xi)\xi)-(n+1)v(\xi)]}\\
        &=-(n+1)\int_{B}{v}+\int_{\partial{B}}{v}.
    \end{aligned}
\end{equation}
On the other hand, we have
\begin{equation}\label{e23}
    \begin{aligned}
      \int_{B}{v}=\int_{\Omega}{v(Du(x))}&=\int_{\Omega}{\langle{Du(x),~x}\rangle-u(x)}\\
      &=\int_{\Omega}{[\text{div}(u(x)x)-(n+1)u(x)]}\\
      &=-(n+1)\int_{\Omega}{u(x)}.
    \end{aligned}
\end{equation}
Combining (\ref{e22}) and (\ref{e23}),~we obtain
\begin{equation}\label{e24}
\int_{\Omega}{-u}=\frac{1}{n(n+2)}\int_{\partial{B}}{v}.
\end{equation}
From (\ref{e21}) and (\ref{e24}),~we deduce
$$\begin{aligned}
    |B|&\leq\frac{1}{n}\int_{\partial{B}}{v}=(n+2)\int_{\Omega}{-u}=(n+2)\int_{\Omega}{-u(\det{D^{2}u})^{\frac{1}{n}}}\\
    &\leq\frac{n+2}{n}\int_{\Omega}(-u)\Delta{u}
    =\frac{n+2}{n}\int_{\Omega}{|Du|^{2}}=\frac{n+2}{n}\int_{B}{|\xi|^{2}}=|B|.
\end{aligned}$$
It follows that
$$\frac{\Delta{u}}{n}=(\det{D^{2}u})^{\frac{1}{n}}=1.$$
Then we conclude that $D^{2}u$ is the identity matrix. We may assume the center of the ball is at $x_0$, which yields the desired conclusion.
\end{proof}

\vspace{3mm}

\section{The cases when $Du(\Omega)=\Omega$ or $Du(\Omega)=\Omega^{*}$}\label{sec4}
In this section, we consider the overdetermined problems corresponding
to the cases $Du(\Omega)=\Omega$ and $Du(\Omega)=\Omega^*$,
respectively. We first give direct proofs of Theorems \ref{GTH1.2} and \ref{GTH1.3}, using a method that combines optimal transport and convex geometry. We then present supplementary $k$-Hessian
extensions for $1\leq k\leq n-1$ by means of the isoperimetric
inequalities for quermassintegrals.
\begin{proof}[Proof of Theorem \ref{GTH1.2}]

Since a convex $C^2(\overline{\Omega})$ solution of $\det D^2u=1$ is automatically strictly convex, hence
\begin{equation}\label{e4.1}
Du(\partial\Omega)=\partial\Omega.
\end{equation}
Since $u=0$ on $\partial\Omega$ and $u$ is strictly convex, $
u<0$ in $\Omega$. Thus $u$ attains its minimum at some point $\bar{x}\in\Omega$, and $
Du(\bar{x})=0$. Since $Du(\Omega)=\Omega$, it follows that $0\in\Omega$. Now let $y\in\partial\Omega$. By (\ref{e4.1}), there exists
$x\in\partial\Omega$ such that
$y=Du(x)$. Therefore, by the boundary condition $|Du|\geq1$ on $\partial\Omega$, we have $
|y|\geq 1$ for every $y\in\partial\Omega$. Since $0\in\Omega$, we claim that $B\subset\Omega$. Indeed, otherwise there exists $y_{0}\in B\setminus\Omega$. By the
convexity of $\Omega$ and the fact that $0\in\Omega$, the segment
$\{ty_{0}:0\leq t\leq1\}$ must meet $\partial\Omega$ at some point
$t_{0}y_{0}$ with $0<t_{0}<1$. But
$|t_0 y_{0}|<1$, contradicting $|y|\ge 1$. Thus $B\subset\Omega$ holds.

By the volume assumption, $
|\Omega|\leq\omega_n=|B|$.
Together with $B\subset\Omega$ this yields
$|B|\leq|\Omega|\leq|B|$, and hence $\Omega=B$. Therefore $
Du(\Omega)=\Omega=B$,
and Theorem \ref{T1.1} applies. We obtain $$
u(x)=\frac{|x-x_0|^2-1}{2}
$$ and $\Omega$ is the unit ball centered at $x_0\in\mathbb{R}^n$. Since  $\Omega=B$, we necessarily have
$x_0=0$. Hence
$$
u(x)=\frac{|x|^2-1}{2},
$$
and $\Omega$ is the unit ball centered at the origin. This completes the proof.
\end{proof}
\begin{proof}[Proof of Theorem \ref{GTH1.3}]
The proof is analogous to that of Theorem \ref{GTH1.2}. Since
$Du(\Omega)=\Omega^{*}$, we have $
Du(\partial\Omega)=\partial\Omega^{*}$. Since $u<0$ in $\Omega$, $u$ attains its minimum at some
$\bar{x}\in\Omega$, and $Du(\bar{x})=0$. Because $Du(\Omega)=\Omega^{*}$, we obtain $0\in\Omega^{*}$. Let $y\in\partial\Omega^{*}$, there exists
$x\in\partial\Omega$ such that $y=Du(x)$. The boundary condition $|Du|\geq1$ therefore gives $
|y|\geq1$ for every $y\in\partial\Omega^{*}$. Since $0\in\Omega^{*}$, we obtain $B\subset\Omega^{*}$.
Indeed, otherwise a segment joining $0$ to a point of
$B\setminus\Omega^{*}$ would meet $\partial\Omega^{*}$ at a point of norm
strictly less than $1$, this leads to a contradiction. By the change of variables formula,
\begin{equation}\label{4.2}
|\Omega^*| = \int_\Omega \det D^2u \, dx = |\Omega| \leq \omega_n.
\end{equation}
Using (\ref{4.2}), we obtain
$$|B|\leq|\Omega^*|=|\Omega|\leq\omega_n=|B|.$$
Therefore $\Omega^*=B$. Consequently, $Du(\Omega)=B$.
Theorem \ref{T1.1} applies and yields $$
u(x)=\frac{|x-x_0|^2-1}{2}$$
for some $x_0\in\mathbb{R}^n$, and $\Omega$ is the unit ball centered at $x_0\in\mathbb{R}^n$, so the desired conclusion follows.
\end{proof}

Theorems \ref{GTH1.2} and \ref{GTH1.3} establish the symmetry result Monge-Amp\`{e}re case. We next consider their extensions to the more general $k$-Hessian equations.
\begin{theorem}\label{T5.1}
 ($Du(\Omega)=\Omega$ case) Let $\Omega$ be a $C^2$ bounded convex domain, $|\Omega|\leq\omega_{n}$ and $u\in C^2(\overline{\Omega})$ be a $k$-convex solution of
\begin{equation}\label{EF3.8}
        \left\{
\begin{array}{ll}
 S_{k}({D^{2}u})=\binom{n}{k}&\text{in}~\Omega,\\
 u=0&\text{on}~\partial{\Omega},\\
|Du|\geq{1}&\text{on}~\partial{\Omega},\\
Du(\Omega)=\Omega,
\end{array}
\right.
\end{equation}
with $k\in\{1,\cdots,n-1\}$, $\omega_{n}$ is the volume of the $n$-dimensional unit ball. Then, up to translation, $u(x)=\dfrac{|x|^2-1}{2}$ and $\Omega$ is the unit ball centered at the origin.
\end{theorem}
\begin{proof}
 By (\ref{F2}), (\ref{F7}) and (\ref{EF3.8}), we have
\begin{equation}\label{F18}
|\Omega|=\int_{\Omega}\frac{S_{k}(D^{2}u)}{\binom{n}{k}}=\int_{\Omega}\frac{(S^{ij}_{k}(D^{2}u)u_{j})_{i}}{k\binom{n}{k}}=\int_{\partial\Omega}\frac{S^{ij}_{k}(D^{2}u)u_{j}u_{i}}{k\binom{n}{k}|Du|}=\int_{\partial\Omega}\frac{H_{k-1}|Du|^{k}}{k\binom{n}{k}}.
\end{equation}
The equality (\ref{F18}), in conjunction with the fact that $|Du|\geq{1}$ on $\partial\Omega$, yields
\begin{equation}\label{F19}
    \frac{1}{k\binom{n}{k}}\int_{\partial\Omega}H_{k-1}\leq |\Omega|.
\end{equation}
On the other hand, by (\ref{F2.9}), we obtain
$$
    \left(\frac{|\Omega|}{\omega_{n}}\right)^{\frac{1}{n}}\leq \left(\frac{W_{n-m}(\Omega)}{\omega_{n}}\right)^{\frac{1}{m}},
$$
for $1\le m \le n$.
Letting $n-m=k$, by \eqref{F4}, we obtain
\begin{equation}\label{F20}
    |\Omega|^{\frac{n-k}{n}}\leq\frac{1}{k\binom{n}{k}\omega_{n}^{\frac{k}{n}}}\int_{\partial\Omega}H_{k-1}.
\end{equation}
Combining (\ref{F19}), (\ref{F20}) and $|\Omega|\leq\omega_{n}$, we obtain
$$\frac{|\Omega|}{{\omega_{n}}^{\frac{k}{n}}}\leq|\Omega|^{\frac{n-k}{n}}\leq\frac{1}{k\binom{n}{k}\omega_{n}^{\frac{k}{n}}}\int_{\partial\Omega}H_{k-1}\leq \frac{|\Omega|}{{\omega_{n}}^{\frac{k}{n}}}.$$
It follows that all the above inequalities become equalities. Thus $|Du|\equiv{1}$ on $\partial\Omega$, and the isoperimetric inequality for quermassintegrals then implies that  $\Omega$ is a ball. Let its radius be $R$, then the volume is $|\Omega|=\omega_{n}R^{n}$, the surface area is $n\omega_{n}R^{n-1}$, and the principal curvatures are all $\frac{1}{R}$. The equality in (\ref{F19}) holds, which implies that  
$$ \frac{1}{k\binom{n}{k}}\int_{\partial\Omega}H_{k-1}=\frac{1}{k\binom{n}{k}}\binom{n-1}{k-1}\left(\frac{1}{R}\right)^{k-1}\cdot n\omega_{n}R^{n-1}=|\Omega|=\omega_{n}R^{n},$$
simplifying, we obtain $\omega_{n}R^{n-k}=\omega_{n}R^{n}$. Hence, we obtain $R = 1$ and $\Omega$ is the unit ball. Up to a translation, we may assume that
$\Omega=B$. We next show that $u$ is radially symmetric. For any orthogonal
matrix $Q$, define
$$
u_Q(x):=u(Qx).
$$
Since $B$ is invariant under orthogonal transformations, $u_Q$
is a $k$-convex solution of
$$
S_k(D^2u_Q)=\binom{n}{k}\quad\text{in }B,
\quad
u_Q=0\quad\text{on }\partial B.
$$
By the uniqueness of $k$-admissible solutions to the $k$-Hessian Dirichlet problem, we have
$u_Q(x)=u(x)$. Hence $u(Qx)=u(x)$ for every orthogonal
matrix $Q$, and therefore
$u$ is radially symmetric. Thus
$$
u(x)=f(r),\quad r=|x|.
$$
For a radial function, the eigenvalues of $D^2u$ are
$f''(r)$ and $f'(r)/r$ with multiplicity $n-1$. Consequently,
$$
S_k(D^2u)
=
\binom{n-1}{k-1}f''(r)
\left(\frac{f'(r)}{r}\right)^{k-1}
+
\binom{n-1}{k}
\left(\frac{f'(r)}{r}\right)^k.
$$
Setting $w(r)=f'(r)$ and using
$$
\binom{n-1}{k}
=\frac{n-k}{k}\binom{n-1}{k-1},
$$
the equation $S_k(D^2u)=\binom{n}{k}$ becomes
$$
\frac{d}{dr}\left(r^{n-k}w(r)^k\right)=nr^{n-1}.
$$
Hence
$$
r^{n-k}w(r)^k=r^n+C.
$$
Since $k\le n-1$ and $u\in C^2(\overline{B})$, we must have
$C=0$. Therefore
$f'(r)=r$, and hence $
f(r)=\frac{r^2}{2}+C_0$.
The boundary condition $f(1)=0$ gives $C_0=-1/2$. Thus
$$
u(x)=\frac{|x|^2-1}{2}.
$$
This completes the proof.
\end{proof}
\begin{remark}
We note that the transport condition $Du(\Omega)=\Omega$ is not directly
used in the proof of Theorem \ref{T5.1}. In fact, the argument shows that
the remaining assumptions already imply that, up to translation,
$\Omega$ is the unit ball and
$u(x)=\frac{|x-x_0|^2-1}{2}$
for some $x_0\in\mathbb{R}^n$. We retain the condition
$Du(\Omega)=\Omega$ in the statement to emphasize the particular
optimal-transport setting considered in this theorem.
\end{remark}
\begin{theorem}\label{T5.2}
($Du(\Omega)=\Omega^{*}$ case) Let $\Omega$ be a $C^2$ bounded convex domain, $|\Omega|\leq\omega_{n}$ and let $\Omega^*\subset\mathbb{R}^n$ be a
bounded domain. If $u\in C^2(\overline{\Omega})$ a $k$-convex solution of
 \begin{equation}\label{EF3.9}
        \left\{
\begin{array}{ll}
 S_{k}({D^{2}u})=\binom{n}{k}&\text{in}~\Omega,\\
 u=0&\text{on}~\partial{\Omega},\\
|Du|\geq{1}&\text{on}~\partial{\Omega},\\
Du(\Omega)=\Omega^{*},
\end{array}
\right.
\end{equation}
with $k\in\{1,\cdots,n-1\}$, $\omega_{n}$ is the volume of the $n$-dimensional unit ball. Then, for some $x_0\in\mathbb{R}^n$, $u(x)=\dfrac{|x-x_0|^2-1}{2}$ and $\Omega$ is the unit ball centered at $x_0$, $\Omega^*$ is the unit ball centered at the origin.
\end{theorem}
\begin{proof}
The argument leading to \eqref{F19}--\eqref{F20} in the proof of
Theorem \ref{T5.1} does not use the condition $Du(\Omega)=\Omega$.
Hence, under the present assumptions, the same argument shows that
$\Omega$ is a unit ball. By translating the coordinates, we may assume that $x_0=0$, so that $\Omega=B$.
The remaining argument in the proof of Theorem \ref{T5.1}, which
uses the rotational invariance of the $k$-Hessian equation and the
uniqueness of the $k$-Hessian Dirichlet problem, shows that
$$
u(x)=\frac{|x|^2-1}{2}.
$$
Undoing the translation, we obtain
$$
u(x)=\frac{|x-x_0|^2-1}{2},
$$
and $\Omega$ is the unit ball centered at $x_0$.
Consequently,
$$
Du(x)=x-x_0
$$
and hence
$$
Du(\Omega)=B.
$$
Since $Du(\Omega)=\Omega^*$ by assumption, it follows that
$$
\Omega^*=B.
$$
This completes the proof.
\end{proof}
\begin{remark}
The argument first yields
$\Omega$ is the unit ball at $x_0$ and $u(x)=\frac{|x-x_0|^2-1}{2}$
for some $x_0\in\mathbb{R}^n$. The transport condition
$Du(\Omega)=\Omega^*$ is then used only to conclude that
$\Omega^*=Du(\Omega)=B$. Thus, we retain this condition in the statement to emphasize the
optimal-transport setting under consideration.
\end{remark}
\vspace{3mm}

\section{Overdetermined problem with a nonstandard boundary condition}\label{sec5}
In this section, we solve a new overdetermined problem with a nonstandard boundary condition, and prove Theorem \ref{TH2} and \ref{TH2-new}.

In order to prove Theorem \ref{TH2} and \ref{TH2-new}, we need the following lemmas.
\begin{lemma}\cite{bib10}\label{LA.1}
Let $\Omega \subset \mathbb{R}^n$ be a bounded convex domain, $n \geq 2$, and $u$ be the strictly convex solution to the problem
     $$ \left\{
\begin{array}{ll}
 \det{D^{2}u}=1&\text{in}~\Omega,\\
u=0&\text{on}~\partial\Omega.
\end{array}
\right.$$
Then the function
$$
\phi(x) = \sum_{k,l=1}^n \frac{\partial S_2(D^2 u)}{\partial u_{kl}} u_k u_l - 2(n-1)u
$$
attains its maximum on the boundary $\partial \Omega$. Moreover, $\phi$ attains its maximum in $\Omega$ if and only if $\Omega$ is an ellipse for $n=2$ or a ball for $n\geq 3$.
\end{lemma}

With Lemma \ref{LA.1} in hand, we are ready to present the proof of Theorem \ref{TH2}.
\begin{proof}[Proof of Theorem \ref{TH2}]
Let us introduce the following auxiliary function
$$\phi(x)=H_{1}|Du|^{3}-2(n-1)u.$$
By (\ref{F7}),~we have
$$\phi(x)=\sum_{i,j=1}^{n}\frac{\partial{S_{2}(D^{2}u)}}{\partial{u_{ij}}}u_{i}u_{j}-2(n-1)u.$$
By Lemma \ref{LA.1}, $\phi$ attains its maximum on the boundary $\partial\Omega$.~Since $\phi=c$ on $\partial\Omega$, we have $\phi\leq c$ in
$\Omega$. By the strong maximum principle, either\\
(i)~$\phi<c$~in~$\Omega$; \\
or
(ii)~$\phi\equiv{c}$~in~$\overline\Omega$.\\
Suppose by contradiction that $\phi$ satisfies (i),~that is
$$H_{1}|Du|^{3}-2(n-1)u<c~~~\text{in}~\Omega.$$
Integrating both sides on $\Omega$,~we obtain
\begin{equation}\label{e16}
\int_{\Omega}{(H_{1}|Du|^{3}-2(n-1)u)}<c|\Omega|.
\end{equation}
 Using $u=0$ on $\partial\Omega$, (\ref{F7}) and (\ref{F2}), we have
\begin{equation}\label{e17}
\int_{\Omega}{H_{1}|Du|^{3}}=\int_{\Omega}{S^{ij}_{2}(D^{2}u)u_{i}u_{j}}=-2\int_{\Omega}{uS_{2}(D^{2}u)},
\end{equation}
Substituting (\ref{e17}) into (\ref{e16}), we obtain
\begin{equation}\label{e18}
    \int_{\Omega}{(H_{1}|Du|^{3}-2(n-1)u)}=-\int_{\Omega}{(2S_{2}(D^{2}u)+2(n-1))u}<c|\Omega|.
\end{equation}
By (\ref{e7}),~we have
$$S_{2}(D^{2}u)\geq\frac{n(n-1)}{2}.$$
Therefore, we have
$$-(2S_{2}(D^{2}u)+2(n-1))u\ge{-(n-1)(n+2)u}.$$
Integrating both sides over $\Omega$,~by (\ref{1.10}), we obtain 
$$-\int_{\Omega}{(2S_{2}(D^{2}u)+2(n-1))u}\ge \int_{\Omega}{-(n-1)(n+2)u}\ge c|\Omega|.$$
This contradicts (\ref{e18}).~Therefore,~it follows that only case (ii) ($\phi\equiv c$ in $\overline \Omega$) holds. By Lemma \ref{LA.1}, which implies that $\Omega$ is an ellipse when $n=2$ and a ball when $n\geq3$.
Moreover, $u$ is elliptically symmetric when $n=2$ and spherically
symmetric when $n\geq3$. .
\end{proof}

Note that the ellipsoidal or spherical symmetry of the solution is not explicitly stated in the Lemma \ref{LA.1}, but it is actually obtained in its proof (see \cite{bib10}). For the derivation details, one may refer to the proof of the following Lemma \ref{L1}. Note that Lemma \ref{L1} provides a new maximum principle for the $P$-function \eqref{P-f} for general $f$ and $\tau=2$.
{\begin{lemma}\label{L1}
    Let $\Omega\subset\mathbb{R}^n$ be a bounded convex domain with
$C^2$ boundary. Assume that $f\in C^2(\mathbb{R})$ is positive and
nondecreasing, and let
$u\in C^4(\Omega)\cap C^2(\overline{\Omega})$ be a convex solution of
$$
\left\{
\begin{array}{ll}
\det D^2u=f(u) & \text{in }\Omega,\\
u=0 & \text{on }\partial\Omega,
\end{array}
\right.
$$
and assume that
$
\frac{\lambda_{\max}(x)}{\lambda_{\min}(x)}
\leq
\left(\frac{3n}{4}\right)^{\frac{n}{2(n-1)}}$ for all $x\in\Omega$, where $\lambda_{\min}(x)$ and $\lambda_{\max}(x)$ denote the smallest
and largest eigenvalues of $D^2u(x)$, respectively. Then the function
    $$\phi(x)=\sum_{k,l=1}^{n}{\frac{\partial{S_{2}(D^2{u})}}{\partial{u_{kl}}}u_{k}u_{l}}-2(n-1)\int_{0}^{u}{f^{\frac{2}{n}}(t)dt}$$
    attains its maximum on $\partial\Omega$. Moreover, $\phi$ attains its maximum in $\Omega$ then $f$ is constant on $u(\Omega)$ and $\Omega$ is an ellipse for $n=2$ or a ball for $n\geq 3$.
\end{lemma}}
\begin{proof}
Let $D^{2}u=(u_{ij})$,~$(u^{ij})=(u_{ij})^{-1}$ and $\frac{\partial{S_{2}(D^{2}u)}}{\partial{u_{kl}}}=b^{kl}$,~then
    $$\phi(x)=\sum_{k,l=1}^{n}{b^{kl}u_{k}u_{l}}-2(n-1)\int_{0}^{u}{f^{\frac{2}{n}}(t)dt}.$$
We will prove the following differential inequality
$$\sum_{i,j=1}^{n}u^{ij}\phi_{ij}\geq{0}~~\text{in}~\Omega.$$
Fix an arbitrary point $x_0\in\Omega$ and choose an orthonormal
coordinate system such that
$$
D^2u(x_0)=\operatorname{diag}(\lambda_1,\ldots,\lambda_n),
$$
where $\lambda_i=u_{ii}(x_0)$ are the eigenvalues of $D^2u(x_0)$. All the calculations below are understood at $x_0$ in this coordinate
system. Let $b^{kl}_{i}=\frac{\partial{b^{kl}}}{\partial{x_{i}}}$,~$b^{kl}_{ii}=\frac{\partial{b^{kl}}}{\partial{x_{i}}\partial{x_{i}}}$.~ By direct computations, we have
$$\begin{aligned}
    \phi_{i}=(\sum_{k,l=1}^{n}{b^{kl}u_{k}u_{l}})_{i}-2(n-1)(\int_{0}^{u}{f^{\frac{2}{n}}(t)dt})_{i}
    =\sum_{k,l=1}^{n}(b^{kl}_{i}u_{k}u_{l}+2b^{ii}u_{ii}u_{i})-2(n-1)f^{\frac{2}{n}}(u)u_{i}
\end{aligned}$$
at $x_0$, and
$$\begin{aligned}
    \phi_{ii}=&\sum_{k,l=1}^{n}(b^{kl}_{i}u_{k}u_{l}+2b^{kl}u_{ki}u_{l})_{i}-2(n-1)(f^{\frac{2}{n}}(u)u_{i})_{i}\\
    =&\sum_{k,l=1}^{n}(b^{kl}_{ii}u_{k}u_{l}+4b^{il}_{i}u_{ii}u_{l}+2{b^{kk}}u_{kii}u_{k})+2b^{ii}u_{ii}^{2}-\frac{4(n-1)}{n}f^{\frac{2-n}{n}}(u)f'(u)u_{i}^{2}-2(n-1)f^{\frac{2}{n}}u_{ii},
    \end{aligned}$$
at $x_0$. Therefore at $x_0$,~we have
$$\begin{aligned}
\sum_{i,j=1}^{n}u^{ij}\phi_{ij}=\sum_{i=1}u^{ii}\phi_{ii}=&\sum_{i,k,l=1}^{n}u^{ii}b^{kl}_{ii}u_{k}u_{l}+4\sum_{i,l=1}^{n}b^{il}_{i}u_{l}+2\sum_{i,k=1}^{n}u^{ii}b^{kk}u_{kii}u_{k}+2\sum_{i=1}^{n}b^{ii}u_{ii}\\
    &-\sum_{i=1}^{n}\frac{4(n-1)}{n}f^{\frac{2-n}{n}}(u)f'(u)u^{ii}u^{2}_{i}-2n(n-1)f^{\frac{2}{n}}(u),
    \end{aligned}$$
    where the identity $u_{ii}(x_0)u^{ii}(x_0) = 1$ is used. Differentiating $\log\det(D^{2}u)=\log f(u)$,~we obtain
\begin{equation}\label{e6}
\sum_{i,j=1}^{n}u^{ij}u_{ijk}=\frac{f'(u)u_{k}}{f(u)}.
\end{equation}
Differentiating once again,~we obtain
\begin{equation}\label{e3}\sum_{i=1}^{n}u^{ii}u_{iikl}=\sum_{i,j=1}^{n}u^{ii}u^{jj}u_{ijk}u_{ijl}+\frac{f(u)(f''(u)u_{k}u_{l}+f'(u)u_{kl})-f'^{2}(u)u_{k}u_{l}}{f^{2}(u)},
\end{equation}
at $x_0$. Since $\sum_{i}^{n}b^{il}_{i}=0$,~$\sum_{i}^{n}b^{ii}u_{ii}=2S_{2}(D^{2}u)$ and (\ref{e6}),~we have
\begin{equation}\label{e5}
\begin{aligned}
\sum_{i=1}^{n}u^{ii}\phi_{ii}=&\sum_{i,k,l=1}^{n}u^{ii}b^{kl}_{ii}u_{k}u_{l}+2\sum_{k=1}^{n}b^{kk}\frac{f'(u)u^{2}_{k}}{f(u)}+4S_{2}(D^{2}u)-2n(n-1)f^{\frac{2}{n}}(u)\\
    &-\frac{4(n-1)}{n}\sum_{i=1}^{n}f^{\frac{2-n}{n}}(u)f'(u)u^{ii}u^{2}_{i}.
\end{aligned}
\end{equation}
Since $b^{kl}=\sum_{j\neq{k}}u_{jj}$ for $k=l$ and $b^{kl}=-u_{kl}$ for $k\neq{l}$, we have $ b^{kl}_{ii}=\sum_{j\neq{k}}u_{jjii}$ for $k=l$ and 
$b^{kl}_{ii}=-u_{klii}$ for $k\neq{l}$. Using (\ref{e3}), we obtain
\begin{equation}\label{e4}
    \begin{aligned}
       \sum_{i,k,l=1}^{n}u^{ii}b^{kl}_{ii}u_{k}u_{l}=&\sum_{\substack{i,j,l=1\\j\neq{l}}}^{n}u^{ii}u_{jjii}u^{2}_{l}-\sum_{\substack{i,k,l=1\\k\neq{l}}}^{n}u^{ii}u_{klii}u_{k}u_{l}\\=&\sum_{\substack{i,j,k,l=1\\k\neq{l}}}^{n}u^{ii}u^{jj}u^{2}_{ijk}u^{2}_{l}-\sum_{\substack{i,j,k,l=1\\k\neq{l}}}^{n}u^{ii}u^{jj}u_{ijk}u_{ijl}u_{k}u_{l}\\&+\sum_{\substack{k,l=1\\k\neq{l}}}^{n}\frac{f(u)(f''(u)u^{2}_{k}+f'(u)u_{kk})-f'^{2}(u)u^{2}_{k}}{f^{2}(u)}u^{2}_{l}\\&-\sum_{\substack{k,l=1\\k\neq{l}}}^{n}\frac{f(u)(f''(u)u_{k}u_{l}+f'(u)u_{kl})-f'^{2}(u)u_{k}u_{l}}{f^{2}(u)}u_{k}u_{l}\\
       =&\sum_{\substack{i,j,k,l=1\\k\neq{l}}}^{n}(u^{ii}u^{jj}u^{2}_{ijk}u^{2}_{l}-u^{ii}u^{jj}u_{ijk}u_{ijl}u_{k}u_{l})+\sum_{\substack{k,l=1\\k\neq{l}}}^{n}\frac{f'(u)}{f(u)}u_{kk}u^{2}_{l}.
    \end{aligned}
\end{equation}
Substituting (\ref{e4}) into (\ref{e5}),~we obtain
\begin{equation}\label{A.5}
\begin{aligned}
    \sum_{i=1}^{n}u^{ii}\phi_{ii}=&\sum_{i,j,k,l,k\neq{l}}^{n}(u^{ii}u^{jj}u^{2}_{ijk}u^{2}_{l}-u^{ii}u^{jj}u_{ijk}u_{ijl}u_{k}u_{l})-\frac{4(n-1)}{n}\sum_{i=1}^{n}f^{\frac{2-n}{n}}(u)f'(u)u^{ii}u^{2}_{i}\\
    &+\frac{3f'(u)}{f(u)}\sum_{k=1}^{n}b^{kk}u^{2}_{k}+[4S_{2}(D^{2}u)-2n(n-1)f^{\frac{2}{n}}(u)],
\end{aligned}\end{equation}
where we have used
$$\sum_{\substack{k,l=1\\k\neq{l}}}^{n}\frac{f'(u)}{f(u)}u_{kk}u_{l}^{2}=\frac{f'(u)}{f(u)}\sum_{k=1}^{n}b^{kk}u_{k}^{2}.$$
By (\ref{e7})
\begin{equation}\label{e8}
S_{2}(D^{2}u)\geq\frac{n(n-1)}{2}(\det{D^{2}u})^{\frac{2}{n}}=\frac{n(n-1)}{2}f^{\frac{2}{n}}(u).
\end{equation}
By the Cauchy-Schwarz's inequality, we have
\begin{equation}\label{e9}
\sum_{k,l=1}^{n}u^{ii}u^{jj}(u^{2}_{ijk}u_{l}^{2}-u_{ijk}u_{ijl}u_{k}u_{l})\geq{0}.
\end{equation}
Next,~we claim that
\begin{equation}\label{e10}
\frac{3f'(u)}{f(u)}\sum_{k=1}^{n}b^{kk}u^{2}_{k}-\frac{4(n-1)}{n}\sum_{i=1}^{n}f^{\frac{2-n}{n}}(u)f'(u)u^{ii}u^{2}_{i}\geq{0}.
\end{equation}
Since $f'(u)\geq{0}$ and $f(u)>0$,~it suffices to prove that
\begin{equation}\label{e11}
\begin{aligned}
3\sum_{k=1}^{n}b^{kk}u^{2}_{k}-\frac{4(n-1)}{n}\sum_{i=1}^{n}f^{\frac{2}{n}}(u)u^{ii}u^{2}_{i}\geq{0}.
\end{aligned}
\end{equation}
Using $\frac{\lambda_{max}}{\lambda_{min}}\leq(\frac{3n}{4})^{\frac{n}{2(n-1)}}$, we have

$$\begin{aligned}
  3\sum_{k=1}^{n}b^{kk}u^{2}_{k}-\frac{4(n-1)}{n}\sum_{i=1}^{n}f^{\frac{2}{n}}(u)u^{ii}u^{2}_{i}&=\sum_{k=1}^{n}\left(3({S_{1}(D^{2}u)-u_{kk}})-\frac{4(n-1)}{n}f^{\frac{2}{n}}(u)\frac{1}{u_{kk}}\right)u^{2}_{k}\\
  &\geq\sum_{k=1}^{n}\left(3(n-1)\lambda_{min}-\frac{4(n-1)}{n}\frac{({\lambda_{max}})^{\frac{2(n-1)}{n}}}{({\lambda_{min}})^\frac{n-2}{n}}\right)u^{2}_{k}\\
  &\geq\sum_{k=1}^{n}\left(3(n-1)\lambda_{min}-\frac{4(n-1)}{n}\frac{3n}{4}\lambda_{min}\right)u^{2}_{k}=0,
\end{aligned}$$
thus we obtain (\ref{e11}). Combining (\ref{e8}),~(\ref{e9}) and (\ref{e10}),~we obtain
$$
\sum_{i=1}^{n}u^{ii}\phi_{ii}\geq{0}.
$$
Hence, by the maximum principle, $\phi$ attains its maximum on
$\partial\Omega$. Furthermore, if $\phi$ attains its maximum at an
interior point of $\Omega$, then $\phi$ is a constant in $\Omega$.

\noindent
Suppose first that $n=2$ and that $\phi$ attains its maximum at an
interior point of $\Omega$. Then, for any point $x_0\in\Omega$, we have
$$
\sum_{i=1}^2 u^{ii}\phi_{ii}=0.
$$
Since all the terms in the decomposition \eqref{A.5} are
nonnegative, each of them must vanish. In particular, since
$S_2(D^2u)=\det D^2u=f(u)$, the last term in \eqref{A.5} vanishes
identically, and hence
$$
-2f'(u)\sum_{i=1}^2u^{ii}u_i^2
+3\frac{f'(u)}{f(u)}
\sum_{k=1}^2b_{kk}u_k^2=0.
$$
For any point $x_0\in\Omega$, we have
$$
D^2u=
\begin{pmatrix}
u_{11} & 0\\
0 & u_{22}
\end{pmatrix},
\quad
u^{11}=\frac1{u_{11}},
\quad
u^{22}=\frac1{u_{22}},
$$
and
$$
b_{11}=u_{22},
\quad
b_{22}=u_{11},
\quad
f(u)=u_{11}u_{22}.
$$
Consequently,
$$
\begin{aligned}
&-2f'(u)\sum_{i=1}^2 u^{ii}u_i^2
+3\frac{f'(u)}{f(u)}
\sum_{k=1}^2 b_{kk}u_k^2
\\
&\quad
=
-2f'(u)
\left(
\frac{u_1^2}{u_{11}}
+\frac{u_2^2}{u_{22}}
\right)
+
3\frac{f'(u)}{u_{11}u_{22}}
\left(
u_{22}u_1^2+u_{11}u_2^2
\right)
\\
&\quad
=
f'(u)
\left(
\frac{u_1^2}{u_{11}}
+\frac{u_2^2}{u_{22}}
\right)
=
f'(u)\sum_{i=1}^2u^{ii}u_i^2.
\end{aligned}
$$
Hence
\begin{equation}\label{A10-new}
f'(u)\sum_{i=1}^2u^{ii}u_i^2=0
\quad\text{for any }x_0\in\Omega.
\end{equation}
Since $D^2u$ is positive definite, we have $u^{11}>0$ and $u^{22}>0$. Moreover, strict convexity and the boundary condition $u=0$ on
$\partial\Omega$ imply that $u$ has a unique critical point
$x^{*}\in\Omega$. Therefore,
$$
|Du(x_0)|>0
\quad\text{for any }x_0\in\Omega\setminus\{x^{*}\},
$$and hence, by \eqref{A10-new},
$$
f'(u(x_0))=0
\quad\text{for any }x_0\in\Omega\setminus\{x^{*}\}.
$$
By continuity,
$$
f'(u(x^{*}))=0.
$$
Thus
$$
f'(u(x_0))=0
\quad\text{for any }x_0\in\Omega.
$$
Consequently, $f$ is constant on the range $u(\Omega)$.
We next analyze the equality case in the inequality
\eqref{e9}. By Cauchy-Schwarz inequality, we have
$$\sum_{i,j=1}^{2} u^{ii} u^{jj} \left( u_{ij1}u_2 - u_{ij2}u_1 \right)^2=0.$$
Therefore,
\begin{equation}\label{A11-new}
\begin{cases}
u_{111}u_2-u_{112}u_1=0,\\
u_{122}u_2-u_{222}u_1=0,\\
u_{112}u_2-u_{122}u_1=0.
\end{cases}
\end{equation}
On the other hand, using $f'(u)=0$, we obtain from \eqref{e6}
\begin{equation}\label{A12-new}
\begin{cases}
u_{111}u^{11}+u_{122}u^{22}=0,\\
u_{112}u^{11}+u_{222}u^{22}=0.
\end{cases}
\end{equation}
We now show that all third-order derivatives of $u$ vanish. Fix
$x_0\in\Omega\setminus\{x^*\}$. Since $|Du(x_0)|>0$, at least one of
$u_1$ and $u_2$ is nonzero. Suppose first that $u_1\neq0$. Then the
first and third equations in \eqref{A11-new} give
$$
u_{112}=\frac{u_2}{u_1}u_{111},
\qquad
u_{122}=\frac{u_2}{u_1}u_{112}
=\frac{u_2^2}{u_1^2}u_{111}.
$$
Substituting these identities into the first equation in
\eqref{A12-new}, we obtain
$$
\left(
u^{11}+\frac{u_2^2}{u_1^2}u^{22}
\right)u_{111}=0.
$$
Since $u^{11},u^{22}>0$, it follows that
$$
u_{111}=u_{112}=u_{122}=0.
$$
The second equation in \eqref{A11-new} then yields $u_{222}=0$. If $u_1=0$, then $u_2\neq0$. In this case, the first and third equations
in \eqref{A11-new} directly give
$$
u_{111}=u_{122}=u_{112}=0,
$$
while the second equation in \eqref{A12-new} gives
$$
u^{22}u_{222}=0.
$$
Since $u^{22}>0$, we again obtain $u_{222}=0$. Therefore,
$$
u_{111}=u_{112}=u_{122}=u_{222}=0
\qquad\text{in }\Omega\setminus\{x^*\}.
$$
By continuity, we further have
$$
D^3u\equiv0\quad\text{in }\Omega.$$
Thus $D^2u$ is a constant positive definite matrix. Hence there exist
a symmetric positive definite matrix $A$, a point ${x^{*}}\in\mathbb R^2$,
and a constant $C\in\mathbb R$ such that
$$
u(x)=\frac12\langle A(x-{x^{*}}),x-{x^{*}}\rangle+C.
$$
Since $u=0$ on $\partial\Omega$ and $u<0$ in $\Omega$, it follows that
$$
\Omega
=
\left\{
x\in\mathbb R^2:
\langle A(x-{x^{*}}),x-{x^{*}}\rangle<-2C
\right\}.$$
Therefore, $\Omega$ is an ellipse and $u$ is elliptically symmetric.

\noindent
Suppose now that $n\geq 3$ and that $\phi$ attains its maximum at an
interior point of $\Omega$. Then, for any $x_0\in\Omega$, we have
$$
\sum_{i=1}^n u^{ii}\phi_{ii}=0.
$$
Since all the terms in the decomposition \eqref{A.5} are
nonnegative, equality must hold in each of them. In particular,
equality holds in the Maclaurin inequality \eqref{e8}. Hence, for
any point $x_0\in\Omega$, we define
$$
\lambda_1(x_0)=\lambda_2(x_0)=\cdots=\lambda_n(x_0)=:\lambda(x_0),
$$
where $\lambda_1(x_0),\ldots,\lambda_n(x_0)$ are the eigenvalues of $D^2u(x_0)$.
Consequently,
$$
D^2u(x_0)=\lambda(x_0) I,
\qquad
\lambda^n(x_0)=\det D^2u(x_0)=f(u(x_0)),
$$
and therefore
$$
\lambda(x_0)=f(u(x_0))^{1/n}.
$$
Thus
$$
u^{ii}(x_0)=\frac1{\lambda(x_0)}=f(u(x_0))^{-1/n},
\qquad
b_{kk}=(n-1)\lambda(x_0)=(n-1)f(u(x_0))^{1/n},
\qquad k=1,\ldots,n.
$$
Hence, for any $x_0\in\Omega$, the expression in \eqref{e10} becomes
$$
\begin{aligned}
&3\frac{f'(u)}{f(u)}
\sum_{k=1}^n b_{kk}u_k^2
-\frac{4(n-1)}{n}
f(u)^{\frac{2-n}{n}}f'(u)
\sum_{i=1}^n u^{ii}u_i^2
\\
&\quad
=
3(n-1)f'(u)f(u)^{\frac1n-1}|Du|^2
-\frac{4(n-1)}{n}
f'(u)f(u)^{\frac{1-n}{n}}|Du|^2
\\
&\quad
=
\frac{(n-1)(3n-4)}{n}
f'(u)f(u)^{\frac1n-1}|Du|^2.
\end{aligned}
$$
Since $f>0$ and $f'\geq 0$, this quantity is nonnegative.
Moreover, equality must hold in \eqref{e10}. Therefore,
$$
f'(u)|Du|^2=0
\quad\text{for any }x_0\in\Omega.
$$
Since $u$ is strictly convex, it has a unique critical point
$x^{*}\in\Omega$. Thus
$$
|Du(x_0)|>0
\quad\text{for any }x_0\in\Omega\setminus\{x^{*}\},
$$
and consequently
$$
f'(u(x_0))=0
\qquad\text{for any }x_0\in\Omega\setminus\{x^{*}\}.
$$
By continuity,
$$
f'(u(x^{*}))=0.
$$
Hence
$$
f'(u(x_0))=0
\quad\text{for any }x_0\in\Omega.
$$
Therefore, $f$ is constant on the range $u(\Omega)$. In particular,
there exists a constant $c_0>0$ such that
$$
f(u)\equiv c_0
\quad\text{in }\Omega.
$$
Consequently,
$$
D^2u=c_0^{1/n}I
\quad\text{in }\Omega.
$$
Thus
$$
u(x)=\frac{c_0^{1/n}}{2}|x-x^{*}|^2+C
$$
for some $x^{*}\in\mathbb R^n$ and $C\in\mathbb R$. Since
$u=0$ on $\partial\Omega$, it follows that $\Omega$ is a ball with radius $R$ and center at $x^{*}$ and $C=-\frac{c_0^{1/n}R^2}{2}$.
Hence $\Omega$ is a ball and $u$ is spherically symmetric.
\end{proof}

With Lemma \ref{L1} in hand, we are ready to show the proof of Theorem \ref{TH2-new}.

\begin{proof}[Proof of Theorem \ref{TH2-new}]
Define
$$F(s):=\int_0^s f(t)^{\frac2n}\,dt$$
and
$$\varphi(x)=
H_1|Du|^3-2(n-1)F(u).
$$
By \eqref{F7},
$$\varphi(x)=
\sum_{i,j=1}^n
\frac{\partial S_2(D^2u)}{\partial u_{ij}}u_i u_j
-2(n-1)\int_0^u f(t)^{\frac2n}\,dt.
$$
Hence, by Lemma \ref{L1}, $\varphi$ attains its maximum on
$\partial\Omega$. Since $u=0$ and $H_1|Du|^3=c$ on $\partial\Omega$, we have
$\varphi=c$ on $\partial\Omega$. Therefore, $\varphi\leq c$ in $\Omega$. By the strong maximum principle in Lemma \ref{L1}, either $\varphi<c$ in $\Omega$ or
$\varphi\equiv c$ in $\Omega$. Suppose, by contradiction, that
$\varphi<c$ in $\Omega$. Integrating over $\Omega$, we obtain
\begin{equation}\label{eq:new-proof-1}
\int_\Omega
\left(
H_1|Du|^3-2(n-1)F(u)
\right)dx
<c|\Omega|.
\end{equation}
Since $u=0$ on $\partial\Omega$, using \eqref{F7} and \eqref{F2}, we obtain
$$
\int_\Omega H_1|Du|^3\,dx
=
\int_\Omega
\sum_{i,j=1}^n
\frac{\partial S_2(D^2u)}{\partial u_{ij}}u_i u_j\,dx \notag=
-2\int_\Omega uS_2(D^2u)\,dx.
$$
Consequently, \eqref{eq:new-proof-1} yields
\begin{equation}\label{eq:new-proof-3}
-2\int_\Omega
\left[
uS_2(D^2u)+(n-1)F(u)
\right]dx
<c|\Omega|.
\end{equation}
On the other hand, since $u$ is convex and $
\det D^2u=f(u)>0$, we have $D^2u>0$ in $\Omega$, and hence $u$ is strictly convex. Since $u=0$ on $\partial\Omega$, it follows that $u<0$ in $\Omega$. By the Maclaurin inequality \eqref{e7},
$$
S_2(D^2u)
\geq
\frac{n(n-1)}{2}
\left(\det D^2u\right)^{\frac2n}
=
\frac{n(n-1)}{2}f(u)^{\frac2n}.
$$
Hence,
$$
-2uS_2(D^2u)\geq
-n(n-1)u f(u)^{\frac2n},
$$
and
$$
-2\left[
uS_2(D^2u)+(n-1)F(u)
\right]
\geq
(n-1)
\left[
-nu f(u)^{\frac2n}-2F(u)
\right].
$$
Integrating over $\Omega$ and using \eqref{1.10-new}, we get
\begin{align}
-2\int_\Omega
\left[
uS_2(D^2u)+(n-1)F(u)
\right]dx
\geq
(n-1)
\int_\Omega
\left[
-nu f(u)^{\frac2n}
-2\int_0^u f(t)^{\frac2n}\,dt
\right]dx \notag\geq c|\Omega|.
\end{align}
This contradicts \eqref{eq:new-proof-3}. Therefore,
$$
\varphi\equiv c\qquad\text{in }\Omega.
$$
By Lemma \ref{L1}, we obtain that $\Omega$ is an ellipse when $n=2$ and a ball when $n\geq3$.
Moreover, $u$ is elliptically symmetric when $n=2$ and spherically
symmetric when $n\geq3$.
\end{proof}
\vspace{3mm}

\section{Conclusion}\label{sec6}
We investigate the overdetermined problem for optimal transportation under the natural boundary conditions $Du(\Omega)=B$, $Du(\Omega)=\Omega$ and $Du(\Omega)=\Omega^{*}$, respectively. In the latter two cases, we impose an additional volume constraint $|\Omega|\leq\omega_n$ and a boundary condition $|Du|\geq1$. A natural question is whether these additional assumptions can be relaxed. A complete understanding of this issue would require a more detailed analysis of the interaction between the transport map and the geometry of the boundary. We leave this interesting question for future investigation.

\vspace{3mm}

\vspace{5mm}

\noindent {\bf Acknowledgements.}
The first author is grateful to Prof. Xiao-Ping Yang for insightful discussions. Both authors also wish to express their sincere gratitude to him for his long-term encouragement and guidance in this research direction.
This work has been supported by the National Natural Science Foundation of China (Grant No. 12271093), the Jiangsu Provincial Scientific Research Center of Applied Mathematics (Grant No. BK20233002), and Shanghai Institute for Mathematics and Interdisciplinary Sciences (SIMIS) under grant number SIMIS-ID-2025-AD.

\bibliographystyle{amsplain}

\end{document}